\documentclass[12pt,reqno]{amsart}
\usepackage{amsfonts, amsmath, amssymb, enumerate}

\advance \topmargin by -\headheight
\advance \topmargin by -\headsep
     
\evensidemargin \oddsidemargin
\newtheorem{theorem}{Theorem}[section]
\newtheorem{lemma}[theorem]{Lemma}

\theoremstyle{definition}

\theoremstyle{remark}

\numberwithin{equation}{section}

\def\pt{{\hspace{1pt}}} \def\mpt{{\hspace{-1pt}}}
\def\llangle{{\langle \! \langle}} \def\rrangle{{\rangle \! \rangle}}

\def\bfa{{\mathbf a}}

\def\bfF{{\mathbf F}}
\def\bfg{{\mathbf g}}

\def\bfx{{\mathbf x}}
\def\bfy{{\mathbf y}}

\def\dbC{{\mathbb C}}
\def\dbF{{\mathbb F}}
\def\dbK{{\mathbb K}}

\def\dbR{{\mathbb R}}
\def\dbZ{{\mathbb Z}}

\def\alp{{\alpha}}
\def\bet{{\beta}} \def\bfbet{{\boldsymbol \beta}}
\def\gam{{\gamma}} 
\def\del{{\delta}} \def\Del{{\Delta}}  
\def\tet{{\theta}} 

\def\lam{{\lambda}} \def\bflam{{\boldsymbol \lambda}}

\def\sig{{\sigma}}

\def\eps{\varepsilon}

\def\le{\leqslant} \def\ge{\geqslant}

\begin{document}
\title[Diophantine inequalities]{Diophantine inequalities and 
quasi-algebraically closed fields}
\author[C. V. Spencer]{Craig V. Spencer}
\address{CVS: Department of Mathematics, Kansas State University, 138 Cardwell 
Hall, Manhattan, KS 66506, U.S.A.}
\email{cvs@math.ksu.edu}
\author[T. D. Wooley]{Trevor D. Wooley$^*$}
\address{TDW: School of Mathematics, University of Bristol, University Walk,
 Bristol BS8 1TW, United Kingdom}
\email{matdw@bristol.ac.uk}
\thanks{$^*$CVS was supported in part by NSF Grant DMS-0635607 and NSA Young 
Investigators Grant H98230-10-1-0155, and TDW first by an NSF grant, and 
subsequently by a Royal Society Wolfson Research Merit Award.}
\date{}
\begin{abstract} Consider a form $g(x_1,\ldots ,x_s)$ of degree $d$, having 
coefficients in the completion $\dbF_q(\mpt (1/t)\mpt )$ of the field of 
fractions $\dbF_q(t)$ associated to the finite field $\dbF_q$. We establish that
 whenever $s>d^2$, then the form $g$ takes arbitrarily small values for non-zero
 arguments $\bfx\in \dbF_q[\pt t\pt ]^s$. We provide related results for 
problems involving distribution modulo $\dbF_q[\pt t\pt ]$, and analogous 
conclusions for quasi-algebraically closed fields in general.\end{abstract}
\maketitle

\section{Introduction} A homogeneous polynomial of odd degree, with real 
coefficients, assumes arbitrarily small values at non-zero integral arguments 
provided only that it possesses a number of variables sufficiently large 
compared to its degree. This conclusion of Schmidt \cite{Sch1980} was 
established by means of an argument remarkable both for its ingenuity and its 
sophistication. With a similar assumption on the number of variables, the 
analogous problem of showing that a form of odd degree, with integral 
coefficients, necessarily vanishes, while plainly no harder, turns out to be 
considerably more straightforward (see Birch \cite{Bir1957}). Motivated by 
familiar correspondence philosophies, one anticipates that similar conclusions 
should be accessible in which the role of the integers $\dbZ$ is replaced by 
the polynomial ring $\dbF_q[\pt t\pt ]$, and that of the real numbers $\dbR$ is 
replaced by the Laurent series $\dbF_q(\mpt (1/t)\mpt )$. In this paper we show 
not only that such may be achieved, but that in addition much sharper 
conclusions may be attained with considerable ease. It is our hope that the 
quantitative results recorded herein may shed light on what is to be expected in
 the above classical situation.\par

We begin by introducing some notation. Let $k$ be a field. We say that a zero of
 a polynomial is {\it non-trivial} when it has a non-zero coordinate. We 
refer to a polynomial having zero constant term as a {\it Chevalley polynomial},
 and call a homogeneous polynomial a {\it form}. Associated to $k$ is the 
polynomial ring $k[\pt t\pt]$ and the field of fractions $\dbK=k(t)$. Write 
$\dbK_\infty=k(\mpt (1/t)\mpt )$ for the completion of $k(t)$ at $\infty$. Each 
element $\alp$ in $\dbK_\infty$ may be written in the shape $\alp=\sum_{j\le n}
a_jt^j$ for some $n\in \dbZ$ and coefficients $a_j=a_j(\alp)$ in $k$ $(j\le n)$. 
We define $\text{ord }\alp$ to be the largest integer $j$ for which $a_j(\alp)
\ne 0$. Fixing a real number $\gam$ with $\gam>1$, we then write $\langle \alp 
\rangle$ for $\gam^{\text{ord }\alp}$, and refer to $\langle \alp\rangle$ as the 
{\it magnitude} of $\alp$. In this context we adopt the convention that 
$\text{ord }0=-\infty$ and $\langle 0\rangle =0$. Finally, when $\bfbet=(\bet_1,
\ldots,\bet_n)\in \dbK_\infty^n$, we define $\langle \bfbet \rangle={\displaystyle
{\max_{1\le i\le n}}}\langle \bet_n\rangle $.

\par In this section we concentrate on the situation in which $k$ is a finite 
field $\dbF_q$, deferring to later sections a more general discussion of 
quasi-algebraically closed fields. Our first result is a consequence of Theorem 
\ref{theorem1.1} below.

\begin{theorem}\label{corollary1.2}
Let $k=\dbF_q$, and let $d$ and $s$ be natural numbers with $s>d^2$. Suppose 
that $F(\bfx)\in \dbK_\infty[x_1,\ldots ,x_s]$ is a Chevalley polynomial of degree
 $d$, whose coefficients have magnitude not exceeding the positive number $H$. 
Then, whenever $0<\eps\le \gam^{-d}H$, the inequality $\langle F(\bfx)\rangle <
\eps$ admits a solution $\bfx\in \dbF_q[\pt t\pt ]^s$ with $0<\langle \bfx
\rangle \le (H/\eps)^{d/(s-d^2)}$.
\end{theorem}

The conclusion of Theorem \ref{corollary1.2} may be compared with work of the 
first author \cite{Spe2007}, where a variant of the Davenport-Heilbronn method 
is applied to investigate the solubility of diagonal diophantine inequalities in
 the function field setting. Let $F(\bfx)\in \dbK_\infty [x_1,\ldots,x_s]$ be a 
diagonal form of degree $d$ whose coefficients are not all in 
$\dbF_q(t)$-rational ratio. Suppose also that the characteristic of $\dbF_q$ 
does not divide $d$, and that the corresponding equation $F(\bfx)=0$ has a 
non-trivial solution over $\dbK_\infty^s$ (a local solubility condition). Then as
 a consequence of Theorem 1.1 of \cite{Spe2007}, when $d$ is large and $s\ge 
(4/3+o(1))d\log d$, it follows that for each $\eps>0$, the inequality $\langle 
F(\bfx)\rangle <\eps$ possesses infinitely many primitive solutions $\bfx \in 
\dbF_q[\pt t\pt ]^s$. Our theorem requires a larger number of variables in order
 to be applicable, but in compensation it addresses general homogeneous 
polynomials, and also supplies an upper bound for the smallest non-trivial 
solution. We note that Hsu \cite{Hsu1999}, \cite{Hsu2001} has examined diagonal 
diophantine inequalities for polynomial rings in which the variables are 
restricted to be irreducibles. The conclusions available in this situation 
resemble those of \cite{Spe2007}, save that the number of variables employed is 
rather larger.\par

As we have already noted, the classical analogue of Theorem \ref{corollary1.2}, 
in which $\dbR$ replaces $\dbK_\infty$ and $\dbZ$ replaces $\dbF_q[\pt t\pt ]$, 
is far more difficult to analyse. The results of Schmidt \cite{Sch1980} are 
explicit neither in the number of variables required to guarantee the existence 
of a solution, nor in terms of the size of the solutions delivered. Freeman 
\cite{Fre2004} has shown that for a given system of $r$ cubic diophantine 
inequalities in the classical setting, the existence of solutions is assured 
whenever $s>(10r)^{(10r)^5}$, but apparently no explicit conclusions are available
 for general forms of higher degree.\par

We turn next to consider the extent to which the bounds on solutions presented 
in Theorem \ref{corollary1.2} can be considered sharp.

\begin{theorem}\label{corollary1.4}
Let $k=\dbF_q$, and let $d$ and $s$ be natural numbers with $s>d^2$. Then there 
exist arbitrarily large numbers $H$, and forms $F(\bfx;H)\in \dbK_\infty [\bfx]$,
 of degree $d$ in $s$ variables, satisfying the following properties:
\begin{enumerate}[(a)]
\item the coefficients of $F$ each have magnitude not exceeding $H$, and
\item the smallest non-trivial solution $\bfx\in \dbF_q[\pt t\pt]^s$ of the 
inequality $\langle F(\bfx;H)\rangle <1$ satisfies the bound $\langle \bfx
\rangle \ge (\gam^{1-d}H)^{d/(s-d^2)}$.
\end{enumerate}
\end{theorem}

This result, which is a consequence of the more general result recorded in 
Theorem \ref{theorem1.4} below, shows that the conclusion of Theorem 
\ref{corollary1.2} is essentially best possible in circumstances wherein 
$\eps =1$. More general values of $\eps$ may also be addressed via Theorem 
\ref{corollary1.4} by simply rescaling the coefficients of the polynomial $F$. 
Further remarks on such lower bounds are offered in section 4.\par
  
We now turn our attention to problems analogous to those in the classical 
literature concerned with the distribution of polynomial sequences modulo $1$. 
Given $\alp \in \dbK_\infty$, we define $\llangle \alp\rrangle ={\displaystyle{
\min_{x\in k[\pt t\pt ]}}}\langle \alp -x\rangle $. As a special case of Theorem 
\ref{theorem1.6}, we derive the following conclusion.

\begin{theorem}\label{corollary1.7}
Let $k=\dbF_q$, and suppose that $f(x)\in \dbK_\infty [x]$ is a Chevalley 
polynomial of degree $d$. Then for each positive number $N$, there exists a 
non-zero polynomial $x\in k[\pt t\pt]$, with $\langle x\rangle \le N$, for which
 $\llangle f(x)\rrangle <N^{-1/d}$. 
\end{theorem}

Define $\| \alp\|$ for $\alp\in \dbR$ by putting $\|\alp\|=\min_{y\in \dbZ}
|\alp-y|$, so that $\|\cdot\|$ is the classical analogue of $\llangle \cdot 
\rrangle$. Also, let $f(t)\in \dbR[\pt t\pt ]$ be a Chevalley polynomial of 
degree $d$. Then, beginning with work of Vinogradov \cite{Vin1927} in the 
special case $f(t)=\alp t^d$, a host of authors have established estimates of 
the type
$$\min_{1\le n\le N}\|f(n)\| \ll_{d,\eps}N^{\eps-\sig(d)},$$
valid for each positive number $\eps$, in which $\sig(d)$ is a suitable positive
 exponent. The current state of the art is given by the permissible exponents 
$\sig(d)=2^{1-d}$ (Schmidt \cite{Sch1977} for $d=2$, and R. C. Baker 
\cite{Bak1982}, \cite{Bak1992} for $d\ge 3$), and $\sig(d)=S(d)^{-1}$ for a 
certain exponent $S(d)$ with $S(d)\sim 4d^2\log d$ (see Corollary 1.3 of Wooley 
\cite{Woo1992}). The conclusion of Theorem \ref{corollary1.7} is therefore 
rather sharper than conclusions available in the analogous classical situation 
whenever $d>2$. In Theorem \ref{theorem1.6} below, we offer more general 
conclusions. These may be compared with results in Chapter 10 of \cite{Bak1986} 
that address situations in which the polynomials $F_j$ are either quadratic or 
diagonal forms.

\section{Quasi-algebraically closed fields}
Our conclusions extend to cover function fields in which the field of constants 
is any quasi-algebraically closed field. In this context we recall the language 
of Lang \cite{Lan1952}, and introduce some of our own. We say that $k$ is a {\it
 strongly $C_i$-field}, or more briefly a {\it $C_i^*$-field}, when any 
Chevalley polynomial of positive degree $d$ lying in $k[\bfx]$, having more than
 $d^i$ variables, necessarily possesses a non-trivial $k$-rational zero. When 
such a conclusion holds only for forms, we say instead that $k$ is a {\it 
$C_i$-field}. In this terminology, algebraically closed fields such as $\dbC$ 
are $C_0^*$-fields, and from the Chevalley-Warning theorem (see \cite{Che1936} 
and \cite{War1936}) it follows that the finite field $\dbF_q$ having $q$ 
elements is a $C_1^*$-field. Work of Lang \cite{Lan1952} and Nagata 
\cite{Nag1957}, moreover, shows that algebraic extensions of $C_i^*$-fields are 
$C_i^*$, and that a transcendental extension, of transcendence degree $j$, over 
a $C_i^*$-field is $C^*_{i+j}$. The same conclusions hold in the absence of 
asterisk decorations.\par

In this section we recall elements of $C_i$-theory relevant to our 
subsequent arguments.

\begin{lemma}\label{lemma2.1}
Let $k$ be a $C_i^*$-field, and suppose that for $1\le j\le r$, the polynomial 
$g_j(\bfx)\in k[x_1,\ldots ,x_s]$ is Chevalley of degree at most $d$. Suppose 
also that $s>rd^i$. Then the system of equations $g_j(\bfx)=0$ $(1\le j\le r)$ 
possesses a non-trivial $k$-rational solution. When $k$ is merely a $C_i$-field,
 the same conclusion holds provided that the polynomials $g_j$ are forms.
\end{lemma}

\begin{proof} This is Theorem 1b of Nagata \cite{Nag1957}.\end{proof}

Note that when $k$ is a $C_i$-field, then it is a consequence of Lemma 
\ref{lemma2.1} that $k$ is a $C^*_{i+1}$-field. For if $g(\bfx)\in k[x_1,\ldots 
,x_s]$ is a Chevalley polynomial of degree $d$, then one may write $g$ in the 
shape $g(\bfx)=g_1(\bfx)+\ldots +g_d(\bfx)$, where each $g_j$ is homogeneous of 
degree $j$. In particular, the equation $g(\bfx)=0$ has a non-trivial 
$k$-rational solution provided only that the system $g_j(\bfx)=0$ $(1\le j\le 
d)$ has such a solution. But the latter is a system of $d$ simultaneous 
homogeneous equations of degree at most $d$, and by Lemma \ref{lemma2.1} this 
system has a non-trivial $k$-rational solution whenever $s>d^{i+1}$, thereby 
confirming our earlier claim.\par

We say that a form $\Psi(\bfx)\in k[x_1,\ldots ,x_s]$ is {\it normic} when it 
satisfies the property that the equation $\Psi(\bfx)=0$ has only the trivial 
solution $\bfx={\bf 0}$. When such is the case, and the form $\Psi(\bfx)$ has 
degree $d$ and contains $d^i$ variables, then we say that $\Psi$ is {\it normic 
of order $i$}. Plainly, when $k$ is a $C_i$-field, any normic form $\Psi(\bfx)$ 
of degree $d$ can have at most $d^i$ variables. We note also that when 
$k=\dbF_q$, then for each natural number $d$ there exist normic forms of degree 
$d$ possessing precisely $d$ variables. In order to exhibit such a form, 
consider a field extension $L$ of $\dbF_q$ of degree $d$, and examine the norm 
form $\Psi(\bfx)$ defined by considering the norm map from $L$ to $\dbF_q$ with 
respect to a coordinate basis for the field extension of $L$ over $\dbF_q$.\par

When $m$ is a non-negative integer, and $F_1,\ldots ,F_r\in \dbK_\infty [x_1,
\ldots ,x_s]$, it is convenient to define $D_m(\bfF)=D_m(F_1,\ldots ,F_r)$ by 
putting
$$D_m(F_1,\dots ,F_r)=(\deg F_1)^m+\ldots +(\deg F_r)^m.$$ 

\begin{lemma}\label{lemma2.2} Let $k$ be a $C_i^*$-field, and suppose that for 
$1\le j\le r$, the polynomial $g_j(\bfx)\in k[x_1,\ldots ,x_s]$ is Chevalley. 
Suppose also that there are normic forms over $k$ of order $i$ of each positive
 degree. Then whenever $s>D_i(\bfg)$, the system of equations $g_j(\bfx)=0$ 
$(1\le j\le r)$ possesses a non-trivial $k$-rational solution. When $k$ is 
merely a $C_i$-field, the same conclusion holds provided that the polynomials 
$g_j$ are forms.
\end{lemma}

\begin{proof} This is Theorem 4 of Lang \cite{Lan1952} when $k$ is a 
$C_i$-field, whilst the argument of the proof of this theorem delivers the 
desired conclusion also when $k$ is $C_i^*$.\end{proof}

\section{Solving inequalities via $C_i$-theory}
We now apply the theory of $C_i$-fields, due to Lang \cite{Lan1952} and Nagata 
\cite{Nag1957}, so as to bound the solutions of diophantine inequalities over 
function fields $k(t)$.

\begin{theorem}\label{theorem1.1} Let $k$ be a $C_i^*$-field. Suppose that 
$F_j(\bfx)\in \dbK_\infty [x_1,\ldots ,x_s]$ $(1\le j\le r)$ are Chevalley 
polynomials of degree at most $d$, whose coefficients have magnitude not 
exceeding the positive number $H$. Put $\Del=\deg F_1+\ldots +\deg F_r$, and 
suppose that $s>\Del d^i$. Then whenever $0<\eps\le \gam^{-d}H$, the system of 
inequalities
\begin{equation}\label{3.5}
\langle F_j(\bfx)\rangle <\eps\quad (1\le j\le r)
\end{equation}
admits a solution $\bfx\in k[\pt t\pt ]^s$ satisfying $0<\langle \bfx\rangle \le
 (H/\eps)^{rd^i/(s-\Del d^i)}$. The same conclusion holds for $C_i$-fields $k$ when 
the polynomials $F_j$ are forms.
\end{theorem}

\begin{proof} We suppose that $k$ is a $C_i^*$-field, and that for $1\le j\le 
r$, the polynomial $F_j(\bfx)\in \dbK_\infty [x_1,\ldots ,x_s]$ is Chevalley of 
degree $d_j\le d$. Let $H$ be an upper bound for the magnitude of the non-zero 
coefficients occurring in $F_j(\bfx)$ $(1\le j\le r)$, and write $h$ for the 
largest integer for which $\gam^h\le H$. It follows that for $1\le j\le r$, the 
coefficients of $F_j(\bfx)$ each have degree at most $h$. We take $B$ to be a 
non-negative integer to be chosen later, and consider an $s$-tuple $(x_1,\ldots 
,x_s)\in k[\pt t\pt ]^s$ wherein each coordinate $x_n$ has $t$-degree $B$. Put
\begin{equation}\label{2.1}
x_n=y_{n0}+y_{n1}t+\ldots +y_{nB}t^B\quad (1\le n\le s),
\end{equation}
and consider the polynomial obtained by substituting this choice for $\bfx$ into
 $F_j(\bfx)$ $(1\le j\le r)$. Thus, for $1\le j\le r$, we obtain
\begin{equation}\label{2.2}
F_j(\bfx)=\sum_{m\le d_jB+h}G_{jm}(\bfy)t^m,
\end{equation}
where each polynomial $G_{jm}(\bfy)\in k[y_{10},\ldots ,y_{sB}]$ is Chevalley of 
degree at most $d_j$ for $1\le j\le r$. Let $M$ be the least integer for which 
$\gam^M>1/\eps$, so that $\gam^{M-1}\le 1/\eps$. We seek a non-trivial solution 
$\bfy \in k^{s(B+1)}$ to the system of equations
\begin{equation}\label{2.3}
G_{jm}(\bfy)=0\quad (-M<m\le d_jB+h,\, 1\le j\le r).
\end{equation}
In view of (\ref{2.2}), the $s$-tuple $\bfx\in k[\pt t\pt ]^s$, associated to 
$\bfy$ via (\ref{2.1}), provides a non-trivial solution to the system 
(\ref{3.5}).\par

The system (\ref{2.3}) consists of $d_jB+h+M$ equations of degree at most $d_j$,
 for $1\le j\le r$, in $s(B+1)$ variables. Since $k$ is presently supposed to be
 a $C_i^*$-field, we find from the first conclusion of Lemma \ref{lemma2.1} that
 the system (\ref{2.3}) possesses a non-trivial solution $\bfy \in k^{s(B+1)}$ 
whenever
\begin{equation}\label{2.5}
s(B+1)>d^i\sum_{j=1}^r(d_jB+h+M).
\end{equation}
Write $\Del =d_1+\ldots +d_r$. The hypotheses of the statement of the theorem 
permit us to assume that $\eps \le \gam^{-d}H$, which implies that $\gam^{-M}<\eps
 \le \gam^{-d}H<\gam^{h+1-d}$. We therefore have $h+M\ge d$, so that when $s>\Del 
d^i$, the condition (\ref{2.5}) is satisfied for a non-negative integral value 
of $B$ with $(s-\Del d^i)B\le rd^i(h+M)-\Del d^i$. In particular, there exists a
 non-trivial solution $\bfx \in k[\pt t\pt ]^s$ to the system (\ref{3.5}) with 
$$\langle \bfx \rangle^{s-\Del d^i}\le \gam^{-\Del d^i}(\gam^{h+M})^{rd^i}\le 
\gam^{(r-\Del)d^i}(H/\eps)^{rd^i}.$$
Since the lower bound $\Del\ge r$ follows from the hypotheses of the statement 
of the theorem, the first conclusion of Theorem \ref{theorem1.1} now follows. 
The second follows in like manner by making use of the final assertion of Lemma 
\ref{lemma2.1}.
\end{proof}

Theorem \ref{corollary1.2} is an immediate consequence of the last theorem, 
since $\dbF_q$ is a $C_1^*$-field. We remark that when $k$ is a $C_i^*$-field, 
and there are normic forms of order $i$ for each positive degree, then the 
conclusions of Theorem \ref{theorem1.1} may be sharpened. If one makes use of 
Lemma \ref{lemma2.2} in place of Lemma \ref{lemma2.1} in the above argument, 
then one may replace the constraint (\ref{2.5}) by the condition
$$s(B+1)>\sum_{j=1}^r(d_jB+h+M)d_j^i.$$
From here one finds that whenever $s>D_{i+1}(\bfF)$, a solution of the system 
(\ref{3.5}) exists for which $\langle \bfx\rangle \le (H/\eps)^{D_i(\bfF)/(s-D_{i+1}(
\bfF))}$. The same conclusion holds for $C_i$-fields when the polynomials $F_j$ 
are forms.\par

We have already remarked on the paucity of explicit results, in the classical 
rational case, for general homogeneous forms of higher degree. In the diagonal 
situation, on the other hand, much more is known, and one even has available 
reasonable bounds for the size of the smallest non-trivial solutions. Put 
$\rho(8)=15/8$ and $\rho(9)=1$. Also, let $s$ be either $8$ or $9$, and consider
 non-zero real numbers $\lam_1,\dots ,\lam_s$. Then it follows from work of 
Br\"udern \cite{Bru1996} that for each positive number $\eps$, and for any 
exponent $\rho$ exceeding $\rho(s)$, the inequality
$$|\lam_1x_1^3+\ldots +\lam_sx_s^3|<\eps$$
possesses an integral solution $\bfx$ satisfying
\begin{equation}\label{1.2}
0<|\lam_1x_1^3|+\ldots +|\lam_sx_s^3|\ll |\lam_1\ldots \lam_s|^\rho
(1/\eps)^{s\rho-1}.
\end{equation}
Sharper conclusions are available when the coefficients $\lam_i$ are integral. 
Indeed, Br\"udern \cite{Bru1994} shows that in such circumstances the exponent 
$\rho(8)=15/8$ may be replaced by $5/3$. We refer the reader to \cite{PR1967} 
for earlier work on this topic.\par

The argument that we employ to establish Theorem \ref{theorem1.1} is easily 
adapted to provide bounds of the shape (\ref{1.2}), and leads to the following 
conclusion.

\begin{theorem}\label{theorem1.3}
Let $k$ be a $C_i$-field, and let $s$ and $d$ be natural numbers with 
$s>d^{i+1}$. Put $\rho=1/(s-d^{i+1})$. Then whenever $\lam_j\in \dbK_\infty^\times$ 
$(1\le j\le s)$, and
\begin{equation}\label{3.6a}
0<\eps \le \gam^{-d}\langle \bflam \rangle^{1-s/d^{i+1}}\langle \lam_1\ldots \lam_s
\rangle^{1/d^{i+1}},
\end{equation}
the inequality
\begin{equation}\label{3.6}
\langle \lam_1x_1^d+\ldots +\lam_sx_s^d\rangle <\eps 
\end{equation}
possesses a solution $\bfx\in k[\pt t\pt ]^s$ satisfying
$$0<\max_{1\le n\le s}\langle \lam_nx_n^d\rangle <\gam^{d-1}\langle \lam_1\ldots 
\lam_s\rangle^\rho (1/\eps)^{s\rho -1}.$$
\end{theorem}

\begin{proof} We adopt an approach similar to that employed in our proof of 
Theorem \ref{theorem1.1}. Let $k$ be a $C_i$-field. For $1\le j\le s$, put 
$h_j=\text{ord}\,\lam_j$, and let $h=\max\{h_1,\ldots ,h_s\}$. We take $B$ to be 
a non-negative integer to be chosen in due course, and on this occasion we 
consider an $s$-tuple $(x_1,\ldots ,x_s)\in k[\pt t\pt ]^s$ with the property 
that for $1\le j\le s$, the polynomial $x_j$ has $t$-degree $B_j=B+[(h-h_j)/d]$.
 Here, as usual, we write $[\tet]$ for the largest integer not exceeding $\tet$.
 Taking
\begin{equation}\label{2.6a}
x_n=y_{n0}+y_{n1}t+\ldots +y_{nB_n}t^{B_n}\quad (1\le n\le s),
\end{equation}
we obtain the expression
\begin{equation}\label{2.6}
\lam_1x_1^d+\ldots +\lam_sx_s^d=\sum_{m\le dB+h}G_m(\bfy)t^m,
\end{equation}
where each polynomial $G_m(\bfy)\in k[y_{10},\ldots ,y_{sB_s}]$ is homogeneous of 
degree $d$. Let $M$ be the least integer for which $\gam^M>1/\eps$, so that 
$\gam^{M-1}\le 1/\eps$, and put $D=B_1+\ldots +B_s$. We seek a non-trivial 
solution $\bfy\in k^{D+s}$ to the system
\begin{equation}\label{2.7}
G_m(\bfy)=0\quad (-M<d\le dB+h).
\end{equation}
In view of (\ref{2.6}), the $s$-tuple $\bfx\in k[\pt t\pt]^s$, associated to 
$\bfy$ via the relations (\ref{2.6a}), then provides a non-zero solution of the 
inequality (\ref{3.6}).\par

The system (\ref{2.7}) consists of $dB+h+M$ homogeneous equations of degree $d$ 
in $D+s$ variables. Since $k$ is a $C_i$-field, we find from the second 
conclusion of Lemma \ref{lemma2.1} that the system (\ref{2.7}) possesses a 
non-trivial solution $\bfy\in k^{D+s}$ whenever $D+s>(dB+h+M)d^i$. This condition
 is equivalent to the constraint
$$s(B+1)+\sum_{j=1}^s[(h-h_j)/d]>d^{i+1}B+d^i(h+M),$$
which is to say
\begin{equation}\label{3.7}
(s-d^{i+1})B>d^i(h+M)-s-\sum_{j=1}^s[(h-h_j)/d].
\end{equation}
Since $\eps>\gam^{-M}$, the hypothesis (\ref{3.6a}) permits us to assume that 
$$h+M>d+\sum_{j=1}^s(h-h_j)/d^{i+1}.$$
It follows that the condition (\ref{3.7}) is satisfied for a non-negative 
integral value of $B$ with
$$(s-d^{i+1})B\le d^i(h+M)-d^{i+1}-\sum_{j=1}^s[(h-h_j)/d].$$
The last condition is satisfied with a value of $B$ satisfying the condition
\begin{align*}
(s-d^{i+1})(dB+h)&\le d^{i+1}M-d^{i+2}+sh-\sum_{j=1}^s(h-h_j)+s(d-1)\\
&\le d^{i+1}(M-1)+\sum_{j=1}^sh_j+(s-d^{i+1})(d-1).
\end{align*}
In particular, when $s>d^{i+1}$, there exists a solution $\bfx\in k[\pt t\pt]^s$ 
to the inequality (\ref{3.6}) with
$$0<\max_{1\le n\le s}\langle \lam_nx_n^d\rangle^{s-d^{i+1}}\le \langle \lam_1\dots 
\lam_s\rangle (\gam^{d-1})^{s-d^{i+1}}(\gam^{M-1})^{d^{i+1}},$$
and the conclusion of the theorem is now immediate.
\end{proof}   

\section{Lower bounds}
By adapting an argument employed by Cassels \cite{Cas1955} in his work on 
solutions of rational quadratic forms, we are able to derive lower bounds for 
the magnitude of non-trivial solutions of certain diophantine equations over 
$\dbF_q[t]$. Such lower bounds apply also, of course, to the solutions of 
corresponding diophantine inequalities.

\begin{theorem}\label{theorem1.4}
Let $k$ be a $C_i$-field, and suppose that a normic form of degree $d$ exists in
 $k[\bfx]$ with $D$ variables. Suppose also that $r$ is a natural number and 
$s>rdD$. Then there exist arbitrarily large numbers $H$, and systems of forms 
$F_j(\bfx;H)\in \dbK_\infty[x_1,\dots ,x_s]$ $(1\le j\le r)$ of degree $d$, 
satisfying the following properties:
\begin{enumerate}[(a)]
\item the coefficients of $F_1,\dots ,F_r$ each have magnitude not exceeding 
$H$, and
\item the smallest non-trivial solution $\bfx\in k[\pt t\pt]^s$ of the 
simultaneous inequalities $\langle F_j(\bfx;H)\rangle <1$ $(1\le j\le r)$ 
satisfies the bound $\langle \bfx\rangle \ge (\gam^{1-d}H)^{rD/(s-rdD)}$.
\end{enumerate}
\end{theorem}

\begin{proof} We seek polynomials $F_j(\bfx;H)$ $(1\le j\le r)$ having 
coefficients lying in $k[\pt t\pt]$. Given such polynomials, the system of 
inequalities $\langle F_j(\bfx;H)\rangle <1$ $(1\le j\le r)$ has a non-trivial 
solution $\bfx \in k[\pt t\pt ]^s$ if and only if the system of equations 
$F_j(\bfx;H)=0$ $(1\le j\le r)$ has a non-trivial solution $\bfx\in k[\pt 
t\pt]^s$. Let $i$ be a non-negative integer, and consider a $C_i$-field $k$ 
which admits a normic form $\Psi(x_1,\dots ,x_D)$ of degree $d$. Note that 
whenever $\bfx\in k[\pt t\pt ]^D$, it follows that $\Psi(\bfx)=0$ if and only if
 $\bfx={\bf 0}$. For the sake of convenience, write $\Del=rdD$. We define the 
polynomials $\Phi_m(\bfx)\in \dbK[x_1,\dots ,x_\Del]$ by putting
\begin{equation}\label{3.0}
\Phi_m(\bfx)=\sum_{j=0}^{d-1}t^j\Psi(x_{mdD+jD+1},\dots ,x_{mdD+jD+D})\quad (0\le m
<r),\end{equation}
and observe that the polynomials $\Phi_m(\bfx)$ $(0\le m<r)$ have coefficients 
lying in $k[\pt t\pt]$. In view of our earlier observation, these polynomials 
have the property that, when $\bfx\in k[\pt t\pt]^\Del$, one has $\Phi_m(\bfx)=0$
 $(0\le m<r)$ if and only if $\bfx={\bf 0}$.\par

Now let $s$ be an integer with $s>\Del$, and let $h$ be a natural number. We 
claim that there exists a positive integer $\del$ having the property that 
there exist at least $h\Del (s-\Del +1)$ distinct monic irreducible polynomials 
in $k[\pt t\pt ]$ of degree $\del$. When $k$ has infinitely many elements, our 
claim follows with $\del=1$ by considering polynomials of the shape $t+\lam$, 
with $\lam \in k$. When $k$ is a finite field, on the other hand, then $k$ is 
isomorphic to $\dbF_q$ for some prime power $q$, and so it suffices to consider 
polynomials of degree sufficiently large in terms of $h$, $\Del$ and $s$. We may
 therefore take distinct monic irreducible polynomials
$$\pi_{uwl}\in k[\pt t\pt]\quad (1\le u\le \Del,\, 0\le w\le s-\Del ,\, 1\le 
l\le h),$$
each of degree $\del$. When $1\le u\le \Del$ and $0\le w\le s-\Del$, write
$$\varpi_{uw}=\prod_{1\le l\le h}\pi_{uwl},$$
put
\begin{equation}\label{3.1}
a_{uv}=\prod_{\substack{0\le w\le s-\Del \\ w\ne v}}\varpi_{uw}\quad (1\le u\le \Del,\, 
0\le v\le s-\Del ),
\end{equation}
and consider the linear forms
\begin{equation}\label{3.2}
L_u(\bfx)=a_{u0}x_{s-\Del+u}+\sum_{v=1}^{s-\Del }a_{uv}x_v\quad (1\le u\le \Del ).
\end{equation}
An examination of the definitions (\ref{3.1}) and (\ref{3.2}) reveals that 
whenever $L_u(\bfx)=0$, then necessarily $\varpi_{u0}|x_{s-\Del +u}$ and 
$\varpi_{uv}|x_v$ $(1\le v\le s-\Del )$.\par

We now seek a non-trivial solution $\bfx\in k[\pt t\pt ]^s$ of the system of 
equations
\begin{equation}\label{3.3}
\Phi_m(L_1(\bfx),\ldots ,L_\Del (\bfx))=0\quad (0\le m<r).
\end{equation}
From the discussion in the opening paragraph of this proof, we find that the 
system (\ref{3.3}) has a non-trivial solution $\bfx\in k[\pt t\pt]^s$ if and 
only if the same holds for the system
$$L_1(\bfx)=\ldots =L_\Del (\bfx)=0.$$
This is a system of $\Del$ homogeneous linear equations in the variables $x_1,
\dots ,x_s$. Since, by hypothesis, we have $s>\Del$, this system of equations 
has a non-trivial solution $\bfx\in k[\pt t\pt]^s$. If one were to have 
$x_1=\ldots =x_{s-\Del }=0$, then it would follow from (\ref{3.2}) 
that $x_{s-\Del +u}=0$ for $1\le u\le \Del $. The latter implies that $\bfx={\bf 
0}$, contradicting the non-triviality of $\bfx$. Consequently, there exists an 
integer $v$, with $1\le v\le s-\Del $, for which $x_v\ne 0$. But the conclusion 
of the previous paragraph then implies that $\varpi_{uv}|x_v$ $(1\le u\le \Del)$,
 whence $x_v$ is divisible by the polynomial $\varpi_{1v}\ldots \varpi_{\Del v}$. 
We thus deduce that any non-trivial solution of the system (\ref{3.3}) satisfies
\begin{equation}\label{3.4}
\langle \bfx\rangle \ge \langle \varpi_{1v}\ldots \varpi_{\Del v}\rangle 
=(\gam^{\del h})^\Del .
\end{equation}

\par The polynomial $\Psi(y_1,\ldots ,y_D)$ has coefficients from $k$, and the 
magnitude of each of the non-zero coefficients of the linear forms $L_u(\bfx)$ 
is precisely $(\gam^{\del h})^{s-\Del}$. Thus, considered as a polynomial in $\dbK
[\bfx]$ with coefficients lying in $k[\pt t\pt]$, the size of the coefficient of
 greatest magnitude within the system of polynomials
$$\Psi(L_{mdD+jD+1}(\bfx),\ldots ,L_{mdD+jD+D}(\bfx))\quad (0\le m<r,\, 0\le j<d)$$
is at most $(\gam^{\del h})^{d(s-\Del)}$. From (\ref{3.0}), it therefore follows 
that the size of the coefficient of greatest magnitude within the polynomials 
$\Phi_m(\bfx)$ $(0\le m<r)$ is at most $H=\gam^{d-1}(\gam^{\del h})^{d(s-\Del)}$. On 
recalling that $\Del=rdD$, a comparison with (\ref{3.4}) reveals that any 
non-trivial solution of the system (\ref{3.3}) satisfies
$$\langle \bfx\rangle \ge (\gam^{1-d}H)^{\Del/(d(s-\Del))}=(\gam^{1-d}H)^{rD/(s-rdD)}.$$
This completes the proof of the theorem.
\end{proof}

In the finite field $\dbF_q$, there exists a normic form of degree $d$, in $d$ 
variables, for every positive integer $d$. The conclusion of Theorem 
\ref{corollary1.4} therefore follows at once. We note that Cassels 
\cite{Cas1955} has established analogous lower bounds in the classical situation
 for rational zeros of a quadratic form. One should observe, however, that in 
Cassels' work, the integer $D$ may be taken arbitrarily large, owing to the 
existence of definite forms in any given number of variables. There is also 
related work of Masser \cite{Mas1998} concerning integral zeros of quadratic 
polynomials.\par

\section{Oddly $C_i$-fields}
We refer to a polynomial having no monomials of even degree as an {\it odd 
Chevalley polynomial}. Motivated by work of Lang concerning the theory of real 
places (see \S3 of \cite{Lan1953}), we say that $k$ is an {\it oddly 
$C_i^*$}-field when any odd Chevalley polynomial lying in $k[\bfx]$, having more
 than $d^i$ variables, necessarily possesses a non-trivial $k$-rational zero. 
When such holds only for {\it forms} of odd degree, we say instead that $k$ is 
{\it oddly $C_i$}. A field $k$ is called {\it real} if $-1$ cannot be expressed 
as a sum of squares in $k$. The field $k$ is described as {\it real closed} when
 it is maximal with respect to this property in its algebraic closure. Thus, the
 field of real numbers $\dbR$ is both real closed and oddly $C_0^*$. Also, a 
generalisation of the Corollary to Theorem 15 of Lang \cite{Lan1953}\footnote{
Here we have noted the transparent typographic error in the statement of this 
corollary.} shows that the function field $\dbR(t_1,\dots ,t_n)$ is oddly 
$C_n^*$. We briefly sketch below how to establish an odd version of Theorem 
\ref{theorem1.1}.

\begin{theorem}\label{theorem1.5} Modify Theorem \ref{theorem1.1} so that when 
$i\ge 1$, the assumption that $k$ be a $C_i^*$-field is replaced by the 
hypothesis that it be oddly $C_i^*$, and likewise in the absence of asterisk 
decorations. In addition, replace the assumption that $k$ be a $C_0^*$-field by 
the hypothesis that it be real closed. Also, let $F_j(\bfx)\in \dbK_\infty 
[x_1,\dots ,x_s]$ be odd Chevalley polynomials of degree at most $d$. Then, 
under the remaining hypotheses of Theorem \ref{theorem1.1}, one has the same 
conclusions.
\end{theorem}

In our proof of Theorem \ref{theorem1.5}, we can afford to be relatively 
informal, the hard work having already been accomplished. The proof of the 
second conclusion of Theorem \ref{theorem1.5} follows in precisely the same 
manner as that of Theorem \ref{theorem1.1}, substituting Theorems 12 and 15 of 
Lang \cite{Lan1953} in place of Lemma \ref{lemma2.1}. In order to avoid 
hypotheses concerning the existence of normic forms in such an argument, one 
should modify the proof of Theorem 12 of \cite{Lan1953} along the lines of the 
proof of Theorem 1a of Nagata \cite{Nag1957}. For the corresponding conclusion 
on oddly $C_i^*$-fields with $i\ge 1$, one may proceed in like manner. When $k$ 
is real closed, it remains to verify that any system of $r$ odd Chevalley 
polynomials, in more than $r$ variables, possesses a non-trivial zero. This we 
achieve by means of an application of an algebraic version of the Borsuk-Ulam 
theorem. Let $F_j(\bfx)\in k[x_1,\dots ,x_s]$ $(1\le j\le r)$ be odd Chevalley 
polynomials, and suppose that $s>r$. By setting $x_j=0$ for $r+1<j\le s$, we may
 suppose without loss that $s=r+1$. The map $f:k^s\rightarrow k^r$, defined by 
taking $f(\bfx)=(F_1(\bfx),\dots ,F_r(\bfx))$, maps the $r$-dimensional sphere, 
defined by the equation $x_1^2+\dots +x_s^2=1$, into $k^r$. Then by the algebraic
 version of the Borsuk-Ulam theorem (see Knebusch \cite{Kne1982}), there exists 
a point $\bfa\in k^s$ with $a_1^2+\dots +a_s^2=1$ for which $F_j(\bfa)=0$ $(1\le 
j\le r)$. This is achieved, in fact, by finding such a point with $F_j(\bfa)=
F_j(-\bfa)$ $(1\le j\le r)$. Not only does this confirm our earlier assertion, 
but by utilising the discussion surrounding the Corollary to Theorem 15 of Lang 
\cite{Lan1953}, one finds also that a function field in $n$ variables over a 
real closed field is oddly $C_n^*$. \par

The refinements to Theorem \ref{theorem1.1} described in its sequel apply, 
mutatis mutandis, to the conclusions of Theorem \ref{theorem1.5}.\par

\section{Distribution modulo $k[\pt t\pt ]$}
A simple modification of the argument employed in the proof of Theorem 
\ref{theorem1.1} delivers a result on the distribution of polynomials modulo 
$k[\pt t\pt]$.

\begin{theorem}\label{theorem1.6}
Let $k$ be a $C_i^*$-field, and suppose that for $1\le j\le r$, the polynomial 
$F_j(\bfx)\in \dbK_\infty[x_1,\dots ,x_s]$ is Chevalley of degree at most $d$. 
Then for each positive number $N$, the simultaneous inequalities
$$\llangle F_j(\bfx)\rrangle <N^{-s/(rd^i)}\quad (1\le j\le r)$$
possess a non-trivial solution $\bfx \in k[\pt t\pt ]^s$ with $\langle \bfx
\rangle \le N$. When $k$ is a $C_i$-field, the same conclusion holds provided 
that the polynomials $F_j(\bfx)$ are forms.
\end{theorem}

\begin{proof} The proof of Theorem \ref{theorem1.6} is swiftly accomplished by 
means of the argument of the proof of Theorem \ref{theorem1.1}. With the same 
notation as that employed in the latter, we seek a non-trivial solution $\bfy 
\in k^{s(B+1)}$ to the system of Chevalley polynomial equations
\begin{equation}\label{2.9}
G_{jm}(\bfy)=0\quad (-M<m<0,\, 1\le j\le r),
\end{equation}
in place of (\ref{2.3}). In view of (\ref{2.2}), the element $\bfx\in k[\pt 
t\pt ]^s$, associated to $\bfy$ via the relations (\ref{2.1}), provides a 
non-trivial solution of the system of inequalities
\begin{equation}\label{2.10}
\llangle F_j(\bfx)\rrangle \le \gam^{-M}\quad (1\le j\le r).
\end{equation}

\par The system (\ref{2.9}) consists of $M-1$ equations of degree at most $d$, 
for $1\le j\le r$, in $s(B+1)$ variables. Since we may currently suppose $k$ to 
be a $C_i^*$-field, we find from the first conclusion of Lemma \ref{lemma2.1} 
that the system (\ref{2.9}) possesses a non-trivial solution $\bfy\in k^{s(B+1)}$ 
whenever
$$s(B+1)>d^i\sum_{j=1}^r(M-1)=(M-1)rd^i.$$
This condition is satisfied for an integral value of $M$ satisfying $rd^iM\ge 
s(B+1)$, and hence the system (\ref{2.10}) has a non-trivial solution $\bfx\in 
k[\pt t\pt ]^s$ with $\langle \bfx\rangle \le \gam^B$ and $\gam^{-M}\le 
(\gam^{B+1})^{-s/(rd^i)}$. The first conclusion of Theorem \ref{theorem1.6} follows 
on taking $B$ to be the largest non-negative integer satisfying $\gam^B\le N$, 
since then we have $(\gam^{B+1})^{-1}<N^{-1}$. The second conclusion of the theorem
 follows in a similar manner.
\end{proof}

Making use, once again, of the fact that $\dbF_q$ is a $C_1^*$-field, we derive 
the consequence of Theorem \ref{theorem1.6} recorded in Theorem 
\ref{corollary1.7}. Finally, we note that a conclusion analogous to that of 
Theorem \ref{theorem1.6} follows for oddly $C_i^*$-fields, and for oddly 
$C_i$-fields, provided that the polynomials $F_j(\bfx)$ are respectively odd 
Chevalley polynomials, and forms of odd degree.

\bibliographystyle{amsbracket}

\begin{thebibliography}{28}

\bibitem{Bak1982}
R. C. Baker, \emph{Weyl sums and Diophantine approximation}, J. London Math. 
Soc. (2) \textbf{25} (1982), 25--34.

\bibitem{Bak1986}
R. C. Baker, \emph{Diophantine inequalities}, London Math. Soc. Monographs, 
Oxford University Press, Oxford, 1986.

\bibitem{Bak1992}
R. C. Baker, \emph{Correction to: ``Weyl sums and Diophantine approximation''}, 
J. London Math. Soc. (2) \textbf{46} (1992), 202--204. 

\bibitem{Bir1957}
B. J. Birch, \emph{Homogeneous forms of odd degree in a large number of 
variables}, Mathematika \textbf{4} (1957), 102--105.

\bibitem{Bru1994}
J. Br\"udern, \emph{Small solutions of additive cubic equations}, J. London 
Math. Soc. (2) \textbf{50} (1994), 25--42.

\bibitem{Bru1996}
J. Br\"udern, \emph{Cubic Diophantine inequalities, II}, J. London Math. Soc. 
(2) \textbf{53} (1996), 1--18.

\bibitem{Cas1955}
J. W. S. Cassels, \emph{Bounds for the least solutions of homogeneous quadratic 
equations}, Proc. Cambridge Philos. Soc. \textbf{51} (1955) 262--264.

\bibitem{Che1936}
C. Chevalley, \emph{D\'emonstration d'une hypoth\`ese de M. Artin}, Abh. Math. 
Sem. Hamburg Univ. \textbf{11} (1936), 73--75.

\bibitem{Fre2004}
D. E. Freeman, \emph{Systems of cubic Diophantine inequalities}, J. Reine Angew.
 Math. \textbf{570} (2004), 1--46.

\bibitem{Hsu1999}
C.-N. Hsu, \emph{Diophantine inequalities for polynomial rings}, J. Number 
Theory \textbf{78} (1999), 46--61.

\bibitem{Hsu2001}
C.-N. Hsu, \emph{Diophantine inequalities for the non-Archimedean line 
$\dbF_q(\mpt (1/T)\mpt )$}, Acta Arith. \textbf{97} (2001), 253--267.

\bibitem{Kne1982}
M. Knebusch, \emph{An algebraic proof of the Borsuk-Ulam theorem for polynomial 
mappings}, Proc. Amer. Math. Soc. \textbf{84} (1982), 29--32.

\bibitem{Lan1952}
S. Lang, \emph{On quasi-algebraic closure}, Ann. of Math. (2) \textbf{55} 
(1952), 373--390.

\bibitem{Lan1953}
S. Lang, \emph{The theory of real places}, Ann. of Math. (2) \textbf{57} (1953),
 378--391.

\bibitem{Mas1998}
D. W. Masser, \emph{How to solve a quadratic equation in rationals}, Bull. 
London Math. Soc. \textbf{30} (1998), 24--28. 

\bibitem{Nag1957}
M. Nagata, \emph{Note on a paper of Lang concerning quasi-algebraic closure}, 
Mem. Coll. Sci. Univ. Kyoto, Ser. A. Math. \textbf{30} (1957), 237--241.

\bibitem{PR1967}
J. Pitman and D. Ridout, \emph{Diagonal cubic equations and inequalities}, Proc.
 Roy. Soc. London Ser. A \textbf{297} (1967), 476--502.

\bibitem{Sch1977}
W. M. Schmidt, \emph{On the distribution modulo $1$ of the sequence $\alp n^2+
\bet n$}, Canad. J. Math. \textbf{29} (1977), 819--826.

\bibitem{Sch1980}
W. M. Schmidt, \emph{Diophantine inequalities for forms of odd degree}, Adv. 
in Math. \textbf{38} (1980), 128--151.

\bibitem{Spe2007}
C. V. Spencer, \emph{Diophantine inequalities in function fields}, Bull. London 
Math. Soc. \textbf{41} (2009), 341--353.

\bibitem{Vin1927}
I. M. Vinogradov, \emph{Analytischer Beweis des Satzes \"uber die Verteilung der
 Bruchteile eines ganzen Polynoms}, Bull. Acad. Sci. USSR (6) \textbf{21} 
(1927), 567--578.

\bibitem{War1936}
E. Warning, \emph{Bemerkung zur vorstehenden Arbeit von Herrn Chevalley}, Abh. 
Math. Sem. Hamburg Univ. \textbf{11} (1936), 76--83.

\bibitem{Woo1992}
T. D. Wooley, \emph{On Vinogradov's mean value theorem}, Mathematika \textbf{39}
 (1992), 379--399.

\end{thebibliography}
\providecommand{\bysame}{\leavevmode\hbox to3em{\hrulefill}\thinspace}

\end{document}